\documentclass[11pt]{article}
\usepackage{graphicx} 
\usepackage{amsmath,mathtools,amsthm,fancyhdr,amssymb,bbm, enumerate,float, mathrsfs}
\usepackage{tcolorbox}
\usepackage{appendix}
\tcbuselibrary{breakable}
\usepackage{listings}
\usepackage{color, comment}
\usepackage{accents}
\usepackage{hyperref}
\usepackage{url}
\usepackage{indentfirst}
\hypersetup{
colorlinks,
linkcolor=blue,
anchorcolor=blue,
citecolor=blue
}

\newcommand{\PV}{\mathrm{P.V.}}
\newcommand{\Om}{\Omega}

\newcommand{\m}{\tilde{m}}

\newcommand{\tla}{\tilde{\lambda}}
\newcommand{\TF}{\tilde{F}}
\newcommand{\TE}{\tilde{E}}

\mathtoolsset{showonlyrefs,showmanualtags} 

\newtheorem{proposition}{Proposition}[section]
\newtheorem{lemma}{Lemma}[section]

\newtheorem{theorem}{Theorem}[section]
\newtheorem{corollary}{Corollary}[section]
\theoremstyle{definition}
\newtheorem{remark}{Remark}[section]
\numberwithin{equation}{section}

\usepackage{enumitem}

\title{Global in Time Estimates for Multi-phase Muskat Problem}
\author{Zirui Wang\footnote{School of Mathematical Sciences, Peking University. Email: \href{zrwang@stu.pku.edu.cn}{zrwang@stu.pku.edu.cn}}}
\date{}

\begin{document}

\maketitle

\begin{abstract}
    We establish global-in-time decay estimates for the multi-phase Muskat problem in the case where the density takes exactly \(n+1\) distinct constant values. We first linearize the system around a flat stable configuration, followed by the study of associated linearized operator. The asymptotic behavior at low frequencies of eigenvalues yields the decay rate of $(1+t)^{-s/2-1/4}$ for Wiener norm $\|f\|_s$, which is slower than the classical case, where the decay rate is $(1+t)^{-s+\nu}$. Afterwards we bound the nonlinear term to close the argument.
\end{abstract}

\section{Introduction}
We consider the multi-phase Muskat problem in 2D. In particular, it describes the motion of fluids in incompressible porous media (IPM) \cite{Bear1975DynamicsOF,10.1063/1.1710292} when the density takes exactly 
$n+1$ distinct constant values. The density is carried by a viscous incompressible flow governed by Darcy's law in the field of gravity:
\begin{align}\label{IPM equation}
    \partial_t\rho+u\cdot\nabla\rho=0,\quad u+\nabla p=-(0,\rho),\quad \text{div}u=0,
\end{align}
where gravity points downward in the $y$ direction. Here $\rho$ represents the density, $u$ is the velocity field, $p$ is the pressure. And we have normalized dynamic viscosity, the permeability of the medium, and the gravitational constant in Darcy's law to unity. This equation is an active scalar, where the velocity has the same level of regularity as $\rho$. In $\mathbb{R}^2$ this can be seen from the following Biot-Savart law
\begin{equation}\label{BS Law}
    u=\nabla^{\perp}(-\Delta)^{-1}\partial_x\rho,\quad\nabla^{\perp}=(-\partial_y,\partial_x).
\end{equation}

There have been numerous results concerning the well-posedness of classical Muskat problem ($n=1$). In this case the interface can be denoted by a function $f(x,t)$ satisfies the following
\begin{equation}
    f_t=\frac{\rho_0-\rho_1}{2\pi}\PV\int\frac{y\left(f_x(x)-f_x(x-y)\right)}{y^2+\left(f(x)-f(x-y)\right)^2}dy.
\end{equation}
Here the density $\rho_1$ lies above $\rho_0$. A first approach is to linearize this equation near $f=0$, which leads to
\begin{equation}
    f_t=\frac{\rho_1-\rho_0}{2}\Lambda f
\end{equation}
where $\Lambda f$ is defined as 
$\widehat{\Lambda f}(\xi)=|\xi|\widehat{f}(\xi)$. This shows that in order to obtain stable regime we need $\rho_1<\rho_0$.

In this classical setting, local well-posedness has been established for initial data in $H^3$ in \cite{CordobaContourDynamics}, which was later improved to $H^2$ in \cite{CHENG201632}. Further local results were obtained for the case with viscosity jump in \cite{CCG08,CCG13}, as well as for small or monotone initial data in \cite{Deng2016OnT2,SCRH04}. For the half-plane setting, local well-posedness was established in \cite{Zlato2024The2M} recently. We also mention here that $H^{\frac{3}{2}}$ is in fact a critical space in \cite{AQ22}. For global results, well-posedness was given in \cite{Constantin2010OnTG} when the Wiener norm of initial data $\|f_0\|_1$ is less than $1/3$ for 2D and $1/5$ for 3D (see also \cite{CCGS13}). Based on this, large time decay rate was given in \cite{Patel03062017} of exact $(1+t)^{-s+\nu}$ for Wiener norm $\|f\|_s$. In \cite{CGSV15}, the authors proved local well-posedness for initial data in $W^{2,p},p\in[1,\infty)$, and further gained global well-posedness given the smallness of the slope of initial surface. Besides, global regularity results in critical space $\dot{W}^{1,\infty}$ for small data is given in \cite{Cameron19,Cameron2020GlobalWF}. We also mention related works on self-similar solutions in 2D and 3D \cite{JSNP21,Na2025GlobalSS} as well as singularity formation \cite{CCFG12,Castro2011RayleighTaylorBF,GomezSerrano2013OnTW,Zlato2024The2M111,Zlato2024The2M}, while the question of finite-time singularity formation in the stable regime with bounded slope remains open.

For the multi-phase Muskat problem, local well-posedness has been established in \cite{Bierler2021TheMM} for the case of two interfaces, and later extended in \cite{Bierler2025TheMM} to the setting with different viscosities. In addition, the interfaces are known not to touch each other, which rules out the formation of squirt singularities, see\cite{Crdoba2009AbsenceOS,Gancedo2013AbsenceOS}. This is different from one-phase Muskat problem in which such blowup exists \cite{Castro2013SplashSF,Crdoba2018OnTS}. Very recently, uniform lifespan for local solutions of three-phase Muskat problem was established in \cite{Castro2025ThreephaseMP}. In contrast to the extensive theory available for the classical Muskat problem, the global theory for the multi-phase setting remains less developed. In this paper, we study the large-time decay of global solutions to the multi-phase Muskat problem. We begin by introducing the system of equations in the following.

\subsection{Setting of the problem}
We consider the fluids characterized by $n+1$ different constant densities $\rho_j$ for $j=0,1,\dots,n$. Thus the density is represented by
\begin{equation}\label{def rho}
    \rho(\textbf{x},t)=\rho_j,\quad \text{if}\quad \textbf{x}\in\Om_j,\quad j=0,1,\dots,n,
\end{equation}
where $\Om_j$ are connected regions separated by $n$ interfaces which are characterized by the functions $y=f_j(x,t),j=1,2,\dots,n$. To be specific the regions are defined as
\begin{equation}
    \Om_j=\{\textbf{x}=(x,y)\in\mathbb{R}^2|f_j(x,t)<y<f_{j+1}(x,t)\},\quad j=0,1,\dots,n
\end{equation}
with $f_0(x,t)=-\infty,f_{n+1}(x,t)=+\infty$. We assume that the interfaces are ordered and remain separated, namely
\begin{equation}
    f_k(x,t)<f_j(x,t),\quad\text{for all}\quad k<j.
\end{equation}
Substituting the above ansatz into the IPM system \eqref{IPM equation} we are able to derive the equation for all interfaces as below
\begin{equation}\label{MPM equation}
    \partial_tf_k=-\frac{1}{2\pi}\sum_{j=1}^n\PV(\rho_{j-1}-\rho_j) \int_{\mathbb{R}}\frac{y\left(\partial_xf_j(x-y,t)-\partial_xf_k(x,t)\right)}{y^2+\left(f_j(x-y,t)-f_k(x,t)\right)^2}dy,\quad k=1,2,\dots,n.
\end{equation}
Here, each term in the summation represents the contribution of the \(j\)-th interface to the evolution of the \(k\)-th interface. In other words, the dynamics of \(f_k\) are determined by the nonlocal interactions with all interfaces in the system, weighted by the corresponding density jumps \(\rho_{j-1}-\rho_j\). We left the detailed derivations in Section \ref{Preliminaries}. 
We remark for system \eqref{MPM equation} in order to work in the stable regime and ensure well-posedness, we assume that the lighter fluid lies above the denser one as in \cite{CordobaContourDynamics}:
\begin{equation}
    \rho_k>\rho_j,\quad \text{for all}\quad k<j.
\end{equation}
And the problem is ill-posed if this condition fails, see \cite{CordobaContourDynamics}. We now observe, as in the classical case, that flat interfaces form a stationary solution to \eqref{MPM equation}. Thus, in this stable regime, it is natural to study the dynamics near such a solution.
Motivated by this, we consider perturbations of a flat layered configuration. More precisely, we assume that the \(k\)-th interface is a small perturbation of the horizontal line \(y=d_k\), where
\[
d_k<d_j \qquad \text{for } k<j.
\]
Accordingly, we introduce the perturbation variables
\begin{equation}\label{def g}
    g_k(x,t)=f_k(x,t)-d_k,
\end{equation}
and assume that \(g_k\) remains small in a suitable sense. In the following context, we treat $\rho_k,d_k$ as fixed constants that do not change in our analysis.

We next linearize the system around the steady state. In Fourier variables, the resulting linearized dynamics for the perturbation vector $\vec{g}(x,t)=(g_1(x,t),\dots,g_n(x,t))^{T}$ takes the form
\begin{equation}
    \partial_t\hat{\vec{g}}(\xi,t)=-|\xi|A(\xi)\hat{\vec{g}}(\xi,t),
\end{equation}
where the $n\times n$ matrix $A$ is defined as
\begin{equation}\label{def matrix A}
    A(\xi)=\left(a_{kj}(\xi)\right)_{n\times n},\quad a_{kj}(\xi)=\frac{\rho_{j-1}-\rho_j}{2}e^{-|d_k-d_j||\xi|}.
\end{equation}
The study of the linear matrix \eqref{def matrix A} plays a key role in our analysis and the computation of exact decay rate. We left the analysis in Section \ref{Sec Linear}.

\subsection{Notations}
We consider the following Wiener norm introduced in \cite{Constantin2010OnTG}
\begin{equation*}
    \|f\|_s:=\int_{\mathbb{R}}|\xi|^s|\hat{f}(\xi)|d\xi,
\end{equation*}
as well as the weighted norm for a positive function $w(\xi)$
\begin{equation}
    \|f\|_w:=\int_{\mathbb{R}}w(\xi)|\hat{f}(\xi)|d\xi,
\end{equation}
where $\hat{f}$ is the standard Fourier transform of $f$
\begin{equation*}
    \hat{f}(\xi)=\mathscr{F}(f)(\xi)=\int_{\mathbb{R}}f(x)e^{- ix\xi}dx.
\end{equation*}
We define the convolution of two functions as usual as
\begin{equation*}
    (f*g)(x)=\int_{\mathbb{R}}f(y)g(x-y)dy.
\end{equation*}
For functions $f,g$ defined in $\mathbb{R}$, we take
\begin{equation}\label{eq_notation_Delta}
    \begin{aligned}
        \Delta_yf(x)=\frac{f(x)-f(x-y)}{y},\\
        [f,g]_y(x)=f(x)-g(x-y).
    \end{aligned}
\end{equation}
We use $\tilde{f}$ to denote the initial data, that is
\begin{equation}
    \tilde{f}(x)=f(0,x).
\end{equation}
Finally we use the notation $f_1\lesssim f_2$ if there exists uniform constant $C>0$ such that $f_1\leq Cf_2$. Also $f_1\approx f_2$ means that $f_1\leq Cf_2$ and $f_2\leq Cf_1$.

\subsection{Main results}
Given a priori assumption of the well-defined fluid interface that exists globally in time, we will study its long-time
behavior in this paper. In particular we have the following main. theorem
\begin{theorem}\label{theorem main}
    Suppose $n\geq 2$. Let $\{f_k\}_{1\leq k\leq n}$ be a solution to system \eqref{MPM equation} and $\{g_k\}_{1\leq k\leq n}$ is defined in \eqref{def g}. Further suppose $g_k\in C([0,\infty);H^{s'}(\mathbb{R}))$ for $s'\geq3$ and all $1\leq k\leq n$. Then there exists universal constant $\gamma$ depends on $s',\{d_k\}_{1\leq k\leq n},\{\rho_k\}_{0\leq k\leq n}$, such that if the initial data is small enough to satisfy
    \begin{align}       \sum_{j=1}^{n}\left(\|\tilde{g}_j\|_{L^2}+\|\tilde{g}_j\|_0+\|\tilde{g}_j\|_1\right)\leq\gamma,
    \end{align}
     we have the uniform in time estimate for any $0\leq s<s'-1$:
     \begin{equation}
         \sum_{j=1}^n\|g_j(t)\|_s\lesssim(1+t)^{-\frac{s}{2}-\frac{1}{4}}.
     \end{equation}
\end{theorem}
\begin{remark}
    In the classical one-interface case, global well-posedness and asymptotic behavior have been extensively studied in \cite{Patel03062017}. In particular, the long-time decay rate of a single interface  is given by
    \begin{equation}
        \|f\|_s\sim (1+t)^{-s+\nu}.
    \end{equation}
    for a constant $\nu$ depends on the regularity of initial data. It is notable that this is different from the case \(n\geq 2\), where the decay rate becomes $(1+t)^{-s/2-1/4}$. This distinction is from the spectral structure of the linearized matrix $A(\xi)$. In the one-interface case, \(A(\xi)\) has only one eigenvalue, which behaves like \(1\) as \(\xi\to 0\). By contrast, when \(n\geq 2\), there is still one eigenvalue with behavior \(1\), but the remaining \(n-1\) eigenvalues behave like \(|\xi|\). It is precisely the slower decay in low frequencies that leads to less decay in the multi-phase setting.
\end{remark}
The rest of the paper is organized as follows. In Section \ref{Preliminaries} we rigorously derive the system for multi-phase Muskat problem and prove the $L^2$ maximum principle. In Section \ref{Sec Linear} we analyze the linear operator $A(\xi)$ and derive the asymptotic behavior of its eigenvalues. Finally in Section \ref{Sec nonlinear} we estimate the nonlinear term using the properties of Wiener algebra following \cite{CCG13}. This allows us to complete the proof of Theorem \ref{theorem main}.

\section{Preliminaries}\label{Preliminaries}
In this section, we briefly derive the equation for the multi-phase system \eqref{MPM equation}. By the Biot--Savart law \eqref{BS Law}, the velocity field can be expressed in terms of the density as follows:
\begin{equation}\label{expression of velocity}
    u(\mathbf{x},t)=(-K_1*\partial_{x_1}\rho,K_2*\partial_{x_1}\rho)=(-K_2*\partial_{x_2}\rho,K_2*\partial_{x_1}\rho),
\end{equation}
where we use the notation $\mathbf{x}=(x_1,x_2)$. Here the kernel $K$ is of Calderon-Zygmund type:
\begin{equation}
    K(\mathbf{x})=-\frac{1}{2\pi}\left(\frac{x_2}{|x|^2},\frac{x_1}{|x|^2}\right)=\left(K_1(\mathbf{x}),K_2(\mathbf{x})\right),
\end{equation}
(see \cite{Stein1993Harmonic}). Since $\rho$ satisfies the settings in \eqref{def rho}
 we have
 \begin{equation}
     \nabla\rho=\sum_{j=1}^n(\rho_{j-1}-\rho_{j})(\partial_{x_1}f_j(x_1,t),-1)\delta(x_2-f_j(x_1,t))
 \end{equation}
 where $\delta$ is the 1-D Dirac distribution. Using again the expression of velocity \eqref{expression of velocity} we obtain
 \begin{equation}\label{singular expression velocity}
     u(x,x_2,t)=-\frac{1}{2\pi}\sum_{j=1}^n(\rho_{j-1}-\rho_j)\PV\int_{\mathbb{R}}\frac{\left(y,y\partial_{x}f_j(x-y,t)\right)}{y^2+\left(x_2-f_j(x-y,t)\right)^2}dy.
 \end{equation}
 The equation above is valid when $x_2\neq f_k(x,t)$ and the principal value is taken at infinity (see again \cite{Stein1993Harmonic}). As $x_2$ approaches $f_k(x,t)$ in the normal direction, discontinuity arises since the vorticity is concentrated on the interfaces. Thus for $\epsilon>0$ we define
 \begin{equation}
         u^1\left(x,f_k(x,t),t\right)=\lim_{\epsilon\rightarrow0}u\left(x-\epsilon\partial_xf_k(x,t),f_k(x,t)+\epsilon,t\right),
     \end{equation}
     and
     \begin{equation}
         u^2\left(x,f_k(x,t),t\right)=\lim_{\epsilon\rightarrow0}u\left(x+\epsilon\partial_xf_k(x,t),f_k(x,t)-\epsilon,t\right).
     \end{equation}
     Then it follows
     \begin{align}
         u^1\left(x,f_k(x,t),t\right)=&-\frac{1}{2\pi}\sum_{j=1}^n(\rho_{j-1}-\rho_j)\PV\int_{\mathbb{R}}\frac{\left(y,y\partial_{x}f_j(x-y,t)\right)}{y^2+\left(f_k(x,t)-f_j(x-y,t)\right)^2}dy\\
         &+\frac{1}{2}(\rho_{k-1}-\rho_k)\frac{\partial_xf_k(x,t)\left(1,\partial_xf_k(x,t)\right)}{\sqrt{1+\partial_xf_k(x,t)^2}},
     \end{align}
     and
     \begin{align}
         u^2(x,f_k(x,t),t)=&-\frac{1}{2\pi}\sum_{j=1}^n(\rho_{j-1}-\rho_j)\PV\int_{\mathbb{R}}\frac{(y,y\partial_{x}f_j(x-y,t))}{y^2+(f_k(x,t)-f_j(x-y,t))^2}dy\\
         &-\frac{1}{2}(\rho_{k-1}-\rho_k)\frac{\partial_xf_k(x,t)(1,\partial_xf_k(x,t))}{\sqrt{1+\partial_xf_k(x,t)^2}}.
     \end{align}We further argue that the velocity in tangential directions only moves particles along the line $f_k(x,t)$ and does not make contribution to the dynamics of the interface. Thus it follows
\begin{equation}
    u\left(x,f_k(x,t),t\right)=-\frac{1}{2\pi}\sum_{j=1}^n(\rho_{j-1}-\rho_j)\PV\int_{\mathbb{R}}\frac{\left(y,y\partial_{x}f_j(x-y,t)\right)}{y^2+\left(f_k(x,t)-f_j(x-y,t)\right)^2}dy,
\end{equation}
due to the fact that
\begin{equation}
    \frac{1}{2}(\rho_{k-1}-\rho_k)\frac{\partial_xf_k(x,t)(1,\partial_xf_k(x,t))}{\sqrt{1+\partial_xf_k(x,t)^2}}
\end{equation}
is in tangential directions of the interface $f_k$. Finally, we use the fact that the interface \(f_k\) evolves according to
\begin{equation}\label{dynamic for string}   \left(0,\partial_tf_k(x,t)\right)\cdot\frac{\left(-\partial_xf_k(x,t),1\right)}{\sqrt{1+\partial_xf_k(x,t)^2}}=u\left(x,f_k(x,t),t\right)\cdot\frac{\left(-\partial_xf_k(x,t),1\right)}{\sqrt{1+\partial_xf_k(x,t)^2}}.
\end{equation}
We obtain
\begin{equation}
    \partial_tf_k=-\frac{1}{2\pi}\sum_{j=1}^n\PV(\rho_{j-1}-\rho_j) \int_{\mathbb{R}}\frac{y\left(\partial_xf_j(x-y,t)-\partial_xf_k(x,t)\right)}{y^2+(f_j(x-y,t)-f_k(x,t))^2}dy.
\end{equation}
Therefore we have derived \eqref{MPM equation}.

We then prove the $L^2$ maximum principle for system \eqref{MPM equation}. The result is fundamentally physical since it can be expressed in terms of the conservation law for gravitational potential energy.
 \begin{theorem}\label{energy conservation}
     Let $\{f_k\}_{1\leq k\leq n}$ be solution to \eqref{MPM equation} and $\{g_k\}_{1\leq k\leq n}$ defined in \eqref{def g} satisfies $g_k\in H^3(\mathbb{R}),1\leq k\leq n$. Then the following inequality holds:     \begin{equation}
         \frac{d}{dt}\left(\sum_{j=1}^n(\rho_{j-1}-\rho_j)\|g_j(\cdot,t)\|_{L^2}^2\right)\leq0.
     \end{equation}
 \end{theorem}
 \begin{proof}
     Let
     \begin{equation}
         \Gamma_k(t):=\{(x_1,x_2)\in\mathbb{R}^2|x_2=f_k(x_1,t)\}
     \end{equation}
     to be the graph of $k$-th interface. And we use the following notations for different domains
     \begin{gather}
         \Om_{k}^{\pm}(t)=\{\textbf{x}\in\mathbb{R}^2|\pm d_k<\pm x_2<\pm f_k(x_1,t)\},\\
         \Om_0(t)=\{\textbf{x}\in\mathbb{R}^2| x_2<f_1(x_1,t)\},\quad \Om_{n+1}(t)=\{\textbf{x}\in\mathbb{R}^2| x_2>f_n(x_1,t)\},\\
         \Om_k'=\{\textbf{x}\in\mathbb{R}^2| d_k<x_2<d_{k+1}\},\quad 0\leq k\leq n, \quad\text{with}\quad d_0=-\infty,\quad d_{n+1}=+\infty.
     \end{gather}
     Thus we have
     \begin{align}
         \int_{\mathbb{R}}g_k(x,t)\partial_tg_k(x,t)dx&=\int_{\Gamma_f}\frac{(x_2-d_k)\partial_tf_k(x_1,t)}{\sqrt{1+\partial_{x}f_k(x_1,t)^2
        }}d\sigma(\textbf{x})\\
        &=\int_{\Gamma_f}(x_2-d_k)(0,\partial_tf_k(x_1,t))\cdot\frac{(-\partial_xf_k(x_1,t),1)}{\sqrt{1+\partial_xf_k(x_1,t)^2
        }}d\sigma(\textbf{x})
     \end{align}
We now recall \eqref{dynamic for string}. The second vector on both sides in \eqref{dynamic for string} is the unit outer normal $n(\textbf{x})$ to $\Om_k^{+}$ when $\textbf{x}$ lies on $\partial\Om_k^{+}$, and it equals $-n(\textbf{x})$ to $\Om_k^{-}$ when $\textbf{x}$ lies on $\partial\Om_k^{-}$. The above observation yields
\begin{align}
    \frac{1}{2}\frac{d}{dt}\|g_k(\cdot,t)\|_{L^2}^2=&=\int_{\partial\Om_k^+}(x_2-d_k)u(\textbf{x})\cdot n(\textbf{x})d\textbf{x}-\int_{\partial\Om_k^-}(x_2-d_k)u(\textbf{x})\cdot n(\textbf{x})d\textbf{x} \\
        &=\int_{\Om_k^+}u_2(\textbf{x})d\textbf{x}-\int_{\Om_k^-}u_2(\textbf{x})d\textbf{x}
\end{align}
where we use the incompressible condition $\text{div}(u)=0$. This leads to the following
\begin{align}
    \frac{d}{dt}&\left(\sum_{j=1}^n(\rho_{j-1}-\rho_j)\|g_j(\cdot,t)\|_{L^2}^2\right)=\sum_{j=1}^n(\rho_{j-1}-\rho_j)\left(\int_{\Om_j^+}u_2(\textbf{x})d\textbf{x}-\int_{\Om_j^-}u_2(\textbf{x})d\textbf{x}\right)\\
    &=\sum_{j=1}^n(\rho_{j-1}-\rho_j)\left(\int_{\Om_j^+}u_2(\textbf{x})d\textbf{x}-\int_{\Om_j^-}u_2(\textbf{x})d\textbf{x}\right)+\sum_{j=0}^n\rho_j\int_{\Om_j'}u_2(\textbf{x})d\textbf{x}\\
    &=\int_{\mathbb{R}^2}\rho(\textbf{x})u_2(\textbf{x})d\textbf{x}
        =\int_{\mathbb{R}^2}\rho(\textbf{x})\partial_{x_1}(-\Delta)^{-1}\rho_{x_1}(\textbf{x})d\textbf{x}=-\|\rho_{x_1}\|_{\dot{H}^{-1}(\mathbb{R}^2)}^2 \leq0
\end{align}
Hence we complete the proof of Theorem \ref{energy conservation}.
 \end{proof}

\section{Linear Estimate}\label{Sec Linear}
We linearize system \eqref{MPM equation} assuming that $g_k$ is small enough
\begin{equation}
   \partial_tg_k(x,t)=\sum_{j=1}^n\left(T_{k,j}g_j(x,t)+N_{k,j}(x,t)\right).
\end{equation}
Here $T_{k,j}$ denotes the linear operator and can be written explicitly as below
\begin{gather}\label{def linear operator}
    T_{k,j}v(x,t)=-\frac{\rho_{j-1}-\rho_j}{2\pi}\PV\int\frac{y\partial_xv(x-y,t)}{y^2+(d_k-d_j)^2}dy,\quad k\neq j,\\
    T_{k,k}v(x,t)=-\frac{\rho_{k-1}-\rho_k}{2\pi}\PV\int\frac{\partial_xv(x-y,t)-\partial_xv(x,t)}{y}dy,
\end{gather}
The remaining terms are nonlinear, and we treat them as perturbative terms by recalling \eqref{eq_notation_Delta}
\begin{gather}\label{def nonlinear k neq j}
    N_{k,j}(x,t)=\frac{\rho_{j-1}-\rho_j}{2\pi}\int\frac{y\partial_x[g_k,g_j]_y(x,t)\left((d_k-d_j)^2-(d_k-d_j+[g_k,g_j]_y(x,t))^2\right)}{(y^2+(d_k-d_j)^2)(y^2+(d_k-d_j+[g_k,g_j]_y(x,t))^2)}dy
\end{gather}
for $k\neq j$ and
\begin{equation}\label{def nonlinear kk}
    N_{k,k}(x,t)=\frac{\rho_{k-1}-\rho_k}{2\pi}\int\partial_x\Delta_yg_k(x,t)\frac{\left(\Delta_yg_k(x,t)\right)^2}{1+\left(\Delta_yg_k(x,t)\right)^2}dy.
\end{equation}
In the argument below, we omit the time variable $t$ whenever no confusion is incurred, for simplicity. We now study the linearized system in the Fourier space, since the linear operator can be expressed as a Fourier multiplier. In particular we have
\begin{equation}\label{fourier for linear}
    \widehat{T_{k,j}v}(\xi)=-\frac{\rho_{j-1}-\rho_j}{2}|\xi|e^{-|d_k-d_j||\xi|}\hat{v}(\xi),\quad 1\leq k,j\leq n.
\end{equation}
Therefore we are ready to write the system in Fourier space as
\begin{equation}\label{eq vector g}
    \partial_t\widehat{\vec{g}}(\xi,t)=-|\xi|A(\xi)\widehat{\vec{g}}(\xi,t)+\widehat{\vec{N}}(\xi,t)
\end{equation}
with the definition of linear matrix $A(\xi)$ in \eqref{def matrix A}. The nonlinear term $\vec{N}(x,t)=(N_1,\dots,N_n)^T(x,t)$ is defined as
\begin{equation}
    N_k(x,t)=\sum_{j=1}^nN_{k,j}(x,t).
\end{equation}
We now analyze the behavior of the matrix \(A(\xi)\). To begin with, we present the following lemma, which provides a criterion for its positive definiteness.
\begin{lemma}[Bochner Theorem]\label{Bochner lemma}
    Let $\phi:\mathbb{R}\rightarrow\mathbb{R}$ to be an even function with the condition that $\hat{\phi}(\xi)\geq0$. Then for any given  $x_1,x_2,\dots,x_n\in\mathbb{R}$, the following matrix
    \begin{equation}
        M=(\phi(x_i-x_j))_{n\times n}
    \end{equation}
    is symmetric and positive semi-definite.
\end{lemma}
\begin{proof}
    For a given real vector $v=(v_1,v_2,\dots,v_n)^T$ we compute
    \begin{align}
        v^TMv&=\sum_{j,k=1}^nv_jv_k\phi(x_j-x_k)\\
        &=\sum_{j,k=1}^nv_jv_k\int_{\mathbb{R}}e^{i(x_j-x_k)\xi}\hat{\phi}(\xi)d\xi\\
        &=\int_{\mathbb{R}}\left|\sum_{j=1}^nv_je^{ix_j\xi}\right|^2\hat{\phi}(\xi)d\xi\geq0,
    \end{align}
    which completes the proof.
\end{proof}

The following proposition is at the heart of our linear estimates.
\begin{proposition}\label{prop linear matrix}
    The matrix $A(\xi)$ is diagonalizable for all $\xi$ with $n$ eigenvalues $\lambda_1(\xi)\leq\lambda_2(\xi)\dots\leq\lambda_n(\xi)$ satisfying that:
    \begin{enumerate}[label=(\roman*), ref=(\roman*)]
        \item\label{item:positive-eigenvalues} For all $\xi\neq0$ we have \begin{equation}
            \lambda_k(\xi)>0,\quad \forall 1\leq k\leq n
        \end{equation} 
        and
        \begin{equation}
            \lambda_1(0)=\dots=\lambda_{n-1}(0)=0,\quad \lambda_n(0)=\frac{\rho_0-\rho_n}{2}.
        \end{equation}

        \item\label{item:asymptotic-eigenvalues}Asymptotic behavior: There exists $\epsilon,\alpha>0$ and a permutation \(\tau\) of \(\{1,\dots,n\}\) such that for all $0<\xi<\epsilon$ we have
        \begin{equation}
            \lambda_1(\xi)\geq\alpha\xi
        \end{equation}
        and
        \begin{equation}
            \lim_{\xi\rightarrow+\infty}\lambda_{\tau(k)}(\xi)=\frac{\rho_{k-1}-\rho_k}{2},\quad 1\leq k\leq n.
        \end{equation}
    \end{enumerate}
\end{proposition}
\begin{proof}
    We notice
    \begin{equation}
        A(\xi)=A_0(\xi)Q
    \end{equation}
    with
    \begin{equation}\label{def Q and A0}        A_0(\xi)=(e^{-|d_j-d_k||\xi|})_{n\times n},\quad Q=diag(\frac{\rho_{k-1}-\rho_k}{2})_{1\leq k\leq n } .  \end{equation}
    Then matrix $A$ is similar to $Q^{\frac{1}{2}}A_0Q^{\frac{1}{2}}$. Since $A_0$ is symmetric and $Q^{\frac{1}{2}}$ is a constant diagonal matrix with all components large than zero, we only need to prove \ref{item:positive-eigenvalues} and \ref{item:asymptotic-eigenvalues} for $A_0(\xi)$, with the asymptotic behavior replaced by
    \begin{equation}
        \lambda_n(0)=1,\quad\lim_{\xi\rightarrow+\infty}\lambda_{\tau(k)}(\xi)=1,\quad 1\leq k\leq n.
    \end{equation}
    We compute
    \begin{equation}
        \varphi(x)=e^{-|x|},\quad\widehat{\varphi}(\xi)= \frac{2}{1+\xi^2}.
    \end{equation}
    Then $A_0(\xi)$ is positive semi-definite  due to Lemma \ref{Bochner lemma} by noticing
    \begin{equation}
        A_0(\xi)=(\varphi(d_k\xi-d_j\xi))_{n\times n}.
    \end{equation}
    Furthermore when $\xi=0$, $A_0(\xi)$ is the all-ones matrix. Hence it has $n-1$ zero eigenvalues and one nonnegative eigenvalue equals $n$. When $\xi\neq 0$, it is straightforward to use upper triangularization to compute the determinant of $A_0(\xi)$ by replacing the $k$-th row $R_k$ with $R_k-e^{-(d_k-d_{k-1})|\xi|}R_{k-1}$ for $k=n,n-1,\dots,2$ successively to obtain
    \begin{equation}
        \text{det}(A_0(\xi))=\prod_{j=1}^{n-1}\left(1-e^{-2(d_{j+1}-d_j)|\xi|}\right)>0.
    \end{equation}
    This implies $A_0(\xi)$ is positive definite for $\xi\neq0$, therefore yields the first part of the proposition.

    When $\xi\rightarrow0^+$ we use the following expansion
    \begin{equation}
        A_0(\xi)=J+B\xi+O(\xi^2).
    \end{equation}
    Here
    \begin{equation}
        J=A_0(0)=\textbf{1}\textbf{1}^T,\quad B=(-|d_k-d_j|)_{n\times n},\quad \textbf{1}=(1,\dots,1)^T.
    \end{equation}
    We decompose the vector field as
    \begin{equation}
        \mathbb{R}^n=V\oplus V^{\perp}
    \end{equation}
    where $V$ denotes the eigenspace corresponding to the only nonzero eigenvalue of $J$
    \begin{equation}
        V=\text{span}\{\textbf{1}\}.
    \end{equation}
    Thus it is straightforward to decompose a unit vector $v=(v_1,v_2,\dots,v_n)^T\in V,\|v\|=1$ into
    \begin{equation}
        v=av'+bv'',\quad v'\in V,v''\in V^{\perp}
    \end{equation}
    with the normalization
    \begin{equation}
        a^2+b^2=1,\quad \|v'\|=1,\quad \|v''\|=1.
    \end{equation}
Direct computation shows
\begin{equation}
    v^TA_0v=a^2\left((v')^TJv'+O(\xi)\right)+\xi\left(2ab(v')^TBv''+b^2(v'')^TBv''\right)+O(\xi^2)
\end{equation}
Choosing a large  constant $C$ such that when $a\geq C|\xi|^{\frac{1}{2}}$,
\begin{equation}
    v^TA_0v\geq \frac{n}{2}C^2\xi-C\xi+O(\xi^2)\geq C\xi
\end{equation}
holds for all small $\xi$. Otherwise, if  $a\leq C|\xi|^{\frac{1}{2}}$, we observe   \begin{align}
        (v'')^TBv''&=\sum_{j,k=1}^n-|d_j-d_k|v''_jv''_k\\
        &=2\sum_{j=1}^{n-1}(d_{j+1}-d_j)\left(\sum_{k=1}^jv''_k\right)^2\geq 0
    \end{align}
    where we use $\sum_{j=1}^nv''_j=0$. Note the equality holds if and only if $v''=\textbf{0}$, then there exists $\mu>0$ which depends on $d_j,\rho_j$, such that
    \begin{equation}
        (v'')^TBv''\geq\mu. 
    \end{equation}
    This implies
    \begin{equation}
        v^TA_0v\geq a^2(n+O(\xi))+O(\xi^{\frac{3}{2}})+\frac{\mu}{2}\xi+O(\xi^2)\geq \frac{\mu}{4}\xi
    \end{equation}
    for sufficiently small $\xi$. The above argument shows by choosing $\alpha=\min\{C,\mu/4\}$ and sufficiently small  $\epsilon$ we have for all $0<\xi<\epsilon$,
    \begin{equation}
        \lambda_1(\xi)\geq\alpha\xi.
    \end{equation}

    For the asymptotic behavior of $A_0(\xi)$ at infinity we simply point out that $A_0(\xi)$ tends to identity matrix as $\xi\rightarrow\infty$. This completes the proof of proposition.
\end{proof}
Given the proof of Proposition \ref{prop linear matrix} we are ready to diagonalize matrix $A(\xi)$ by defining
\begin{gather}\label{def diagonalize}
    P^{-1}(\xi)A(\xi)P(\xi)=D(\xi),\\
    D(\xi)=diag(\lambda_k(\xi))_{1\leq k\leq n},\quad P(\xi)=Q^{-\frac{1}{2}}P_0(\xi)
\end{gather}
where $Q$ is defined in \eqref{def Q and A0} and $P_0(\xi)$ can be chosen as an orthogonal matrix. We then define
\begin{equation}\label{def h}
    \widehat{\vec{h}}(\xi,t)=(\hat{h}_1,\dots,\hat{h}_n)^T(\xi,t)=P^{-1}(\xi)\widehat{\vec{g}}(\xi,t).
\end{equation}
Given the asymptotic behavior in the above proposition we can also take an even function $\tla(0)=0$ such that
\begin{equation}\label{def tilde lambda}
    \frac{\alpha}{2}|\xi|\leq\tla(|\xi|)\quad \text{for}\quad \xi\rightarrow0,\quad \tla(\xi)\quad\text{non-decreasing and}\quad \tla(\xi)\leq\lambda_1(\xi) \quad \text{for all}\quad \xi>0.
\end{equation}
\begin{remark}\label{remark gh}
    We remark here that by the definition of $P$ in \eqref{def diagonalize} it is obvious
\begin{equation}
    \sum_{j=1}^n\|\hat{h}_j(\xi)\|\approx\sum_{j=1}^n\|\hat{g}_j(\xi)\|
\end{equation}
for any norm defined in the Fourier space. Thus in the following context we should keep in mind that these two terms are equivalent to deal with.
\end{remark}

\section{Bounds on nonlinear terms}\label{Sec nonlinear}
In this section we present bounds on the  nonlinear terms and close the proof of Theorem \ref{theorem main}.

\subsection{Bounds on the nonlinear term $N_{k,k}$}
We follow the strategy presented in \cite{CCGS13} to give the estimate of $\int|\xi|^s\left|\mathscr{F}N(\xi)\right|$, in which we use the Taylor expansion of $(\Delta_yf(x))^2/\left(1+(\Delta_yf(x))^2\right)$. Then we derive an upper bound for $\left|\int\mathscr{F}(\partial_x\Delta_yf(x)\cdot(\Delta_yf(x))^{2n})(\xi)dy\right|$ for each $n$, and sum over \(n\) to obtain the desired estimate. Repeating the arguments in \cite{Patel03062017} and \cite{CCGS13}, we obtain the following estimate.
\begin{proposition}\label{prop estimate N1}
For $\|g_k\|_{s+1},\|g_k\|_1<\infty$, the following estimates for nonlinear term $N_{k,k}$ holds:
\begin{equation}\label{ineq_prop_estimate_N1}
    \int|\xi|^s\cdot|\mathscr{F}(N_{k,k})(\xi)|d\xi\lesssim \sum\limits_{l=1}^{\infty}(2l+1)^s\left(\|g_k\|_1\right)^{2l}\|g_k\|_{s+1}.
\end{equation}
\end{proposition}
We now state the following lemma as in \cite{CCGS13}:
\begin{lemma}\label{prop_estimate_N1_lemma}
    For each $l$ and function $f$, we have:
    \begin{equation*}
    \begin{aligned}
        &\int|\xi|^s\cdot\left|\int \mathscr{F}(\partial_x\Delta_yf(x)\cdot\left(\Delta_yf(x))^{2l}\right)
        (\xi)dy\right|d\xi\leq\\
        &2\pi\int\int\dots\int|\xi|^s|\xi-\xi_1|\|\hat{f}(\xi-\xi_1)|\times\prod_{j=1}^{2l-1}|\xi_j-\xi_{j+1}||\hat{f}(\xi_j-\xi_{j+1})||\xi_{2l}||\hat{f}(\xi_{2l})|d\xi d\xi_1\dots d\xi_{2l}.
    \end{aligned}
    \end{equation*}
\end{lemma}
\begin{proof}
    We first observe
    \begin{gather}
        \mathscr{F}(\Delta_yf)=m(\xi,y),\quad  \mathscr{F}(\partial_x\Delta_yf)=m(\xi,y)=i\xi m(\xi,y),\\
        m(\xi,y)=\frac{1-e^{-i\xi y}}{y}\hat{f}(\xi)=i\xi\int_0^1e^{iy(z-1)\xi}dz\hat{f}(\xi).
    \end{gather}
    Therefore
    \begin{gather}
        \int \mathscr{F}(\partial_x\Delta_yf(x)\cdot(\Delta_yf(x))^{2l})
        (\xi)dy=\int((i\xi m)*m*\dots*m)(\xi,y)dy\\
        =i\int d\xi_1\dots\int d_{\xi_{2l}}
(\xi-\xi_1)\hat{f}(\xi-\xi_1)\times\left(\prod_{j=1}^{2l-1}(\xi_j-\xi_{j+1})\hat{f}(\xi_j-\xi_{j+1})\right)\xi_{2l}\hat{f}(\xi_{2l})M_l    \end{gather}
with exact $2l$ convolutions. Here the term $M_l(\xi,\xi_1,\dots,\xi_{2l})$ can be computed as
\begin{align}
    M_l(\xi,\xi_1,\dots,\xi_{2l})&=\int_0^1dz_1\dots\int_0^1dz_{2l}\int_{\mathbb{R}}dy\frac{1-e^{-iy(\xi-\xi_1)}}{y}\\
    &\quad\quad \times \text{exp}\left(iy\sum_{j=1}^{2l-1}(z_j-1)(\xi_j-\xi_{j+1}+iy(z_{2l}-1)\xi_{2l})\right)\\
    &=\int_0^1dz_1\dots\int_0^1dz_{2l}\left(\PV\int_{\mathbb{R}}\text{exp}(iyG_1)\frac{dy}{y}-\PV\int_{\mathbb{R}}\text{exp}(iyG_2)\frac{dy}{y}\right)\\
    &=i\pi\int_0^1dz_1\dots\int_0^1dz_{2l}(\text{sgn}G_1-\text{sgn}G_2)
\end{align}
and
\begin{align}
    G_1=\sum_{j=1}^{2l-1}(z_j-1)(\xi_j-\xi_{j+1})+(z_{2l}-1)\xi_{2l}=-\xi_1+\sum_{j=1}^{2l}z_j\xi_j-\sum_{j=1}^{2l-1}z_j\xi_{j+1}.
\end{align}
Additionally
\begin{align}
    G_2&=-(\xi-\xi_1)+\sum_{j=1}^{2l-1}(z_j-1)(\xi_j-\xi_{j+1})+(z_{2l}-1)\xi_{2l}\\
    &=-\xi+\sum_{j=1}^{2l}z_j\xi_j-\sum_{j=1}^{2l-1}z_j\xi_{j+1}.
\end{align}
The above computations clearly show $|M_l|\leq2\pi$, which is enough to close Lemma \ref{prop_estimate_N1_lemma}.
\end{proof}
\begin{proof}[Proof of Proposition \ref{prop estimate N1}]
The Taylor expansion of $(\Delta_yf(x))^2/(1+(\Delta_yf(x))^2)$ can be written as:
\begin{equation*}
    \frac{(\Delta_yf(x))^2}{1+(\Delta_yf(x))^2}=\sum\limits_{n=1}^{\infty}(-1)^{n+1}(\Delta_yf(x))^{2n}.
\end{equation*}
Hence we are ready to bound $N_{k,k}$ as
\begin{align}
   &\int|\xi|^s\cdot|\mathscr{F}(N_{k,k})(\xi)|d\xi=\frac{\rho_{k-1}-\rho_k}{2\pi}\int |\xi|^s\left|\int\mathscr{F}\left(\partial_x\Delta_yg_k(x)\frac{(\Delta_yg_k(x))^2}{1+(\Delta_yg_k(x))^2}\right)dy\right|d\xi.\\
   &\lesssim\sum_{l=1}^{\infty}\int|\xi|^s\cdot\left|\int \mathscr{F}(\partial_x\Delta_yf(x)\cdot(\Delta_yf(x))^{2l})
        (\xi)dy\right|d\xi\\
    &\lesssim\int\int\dots\int|\xi|^s|\xi-\xi_1||\hat{f}(\xi-\xi_1)|\times\prod_{j=1}^{2l-1}|\xi_j-\xi_{j+1}||\hat{f}(\xi_j-\xi_{j+1})||\xi_{2l}||\hat{f}(\xi_{2l})|d\xi d\xi_1\dots d\xi_{2l}
\end{align}
using Lemma \ref{prop_estimate_N1_lemma}. Then the inequality $|\xi|^s\leq(2l+1)^{s-1}(|\xi-\xi_1|^s+\dots+|\xi_{2l}|^s)$ yields
\begin{equation}
    \int|\xi|^s\cdot|\mathscr{F}(N_{k,k})(\xi)|d\xi \lesssim\sum\limits_{l=1}^{\infty}(2l+1)^s(\|g_k\|_1)^{2l}\|g_k\|_{s+1},
\end{equation}
where we use the convolution type estimate
\begin{equation}\label{ineq_convolution}
    \|f_1*f_2\dots*f_n\|_{L^1}\leq\|f_1\|_{L^1}\|f_2\|_{L^1}\dots\|f_n\|_{L^1}.
\end{equation}
\end{proof}
We now take the notation that
\begin{equation}\label{def F 0s}
     F_{0,s}(x)=\sum\limits_{l=1}^{\infty}(2l+1)^sx^{2l}.
\end{equation}
So Proposition \ref{prop estimate N1} also admits the following form
\begin{equation}
    \int|\xi|^s\cdot|\mathscr{F}(N_{k,k})(\xi)|d\xi\lesssim F_{0,s}(\|g_k\|_1)\|g_k\|_{s+1}
\end{equation}

\subsection{Bounds on the nonlinear term $N_{k,j}$}
To derive the estimate for the remaining nonlinear term $N_{k,j}$ when $k\neq j$, we need to give the Taylor expansion of \eqref{def nonlinear k neq j} near the steady state $[g_k,g_j]_y(x)=0$. We now compute \eqref{def nonlinear k neq j} using the simplified notation $d_{k,j}=|d_j-d_k|$
\begin{align}\label{eq N k>j}
    N_{k,j}&=\frac{\rho_{k-1}-\rho_k}{2\pi}\int-\frac{y}{(y^2+d_{k,j}^2)^2}\partial_x[g_k,g_j]_y(x)\frac{[g_k,g_j]_y(x)(2+[g_k,g_j]_y(x))}{1+\frac{2}{y^2+d_{k,j}^2}[g_k,g_j]_y(x)+\frac{1}{y^2+d_{k,j}^2}([g_k,g_j]_y(x))^2}dy\\
    &=\frac{\rho_{k-1}-\rho_k}{2\pi}\int\sum\limits_{l=1}^{\infty}b_{l,k,j}(y)\partial_x[g_k,g_j]_y(x)([g_k,g_j]_y(x))^ldy,\qquad k>j
\end{align}
and
\begin{align}\label{eq N k<j}
    N_{k,j}&=\frac{\rho_{k-1}-\rho_k}{2\pi}\int\frac{y}{(y^2+d_{k,j}^2)^2}\partial_x[g_k,g_j]_y(x)\frac{[g_k,g_j]_y(x)(2-[g_k,g_j]_y(x))}{1-\frac{2}{y^2+d_{k,j}^2}[g_k,g_j]_y(x)+\frac{1}{y^2+d_{k,j}^2}([g_k,g_j]_y(x))^2}dy\\
    &=\frac{\rho_{k-1}-\rho_k}{2\pi}\int\sum\limits_{l=1}^{\infty}b_{l,k,j}(y)\partial_x[g_k,g_j]_y(x)([g_k,g_j]_y(x))^ldy,\qquad k<j.
\end{align}
It follows directly that the coefficients $b_{l,k,j}$ must obey the following
\begin{gather}\label{property of bl}
    \frac{y}{(y^2+d_{k,j}^2)^2}\frac{z(2-z)}{1-\frac{1}{y^2+d_{k,j}^2}z(2-z)}=\sum\limits_{l=1}^{\infty}b_{l,k,j}(y)z^l,\quad k<j,\\
    b_{l,k,j}(y)=-b_{l,k,j}(-y)=(-1)^lb_{l,j,k}(y).
\end{gather}
For all $z>0$ and $z$ close to 0, the inequality $z(2-z)<2z$ shows there exists universal constant $C_0$ such that
\begin{equation}\label{upper bound bl}
    |b_{l,k,j}(y)|\leq2^l\frac{y}{(y^2+d_{k,j}^2)^{l+1}}\leq C_0^l\frac{y}{(y^2+1)^{l+1}}
\end{equation}
With these preliminary estimates established, we now state the bounds for the nonlinear terms $N_{k,j}$ in the proposition below.
\begin{proposition}\label{prop nonlinear kj}
    For $\|g_k\|_{s+1}+\|g_j\|_{s+1},\|g_k\|_0+\|g_j\|_0,\|g_k\|_1+\|g_j\|_1<\infty$, the following estimates hold for nonlinear term $N_{k,j}$:
    \begin{align}
        \int|\xi|^s\cdot|\mathscr{F}(N_{k,j})(\xi)|d\xi&\lesssim F_{1,s}(\|g_k\|_0+\|g_j\|_0)\left(\|g_k\|_{\tla(\xi)|\xi|^{s+1}}+\|g_j\|_{\tla(\xi)|\xi|^{s+1}}\right)\\
            &+(\|g_k\|_1+\|g_j\|_1)F_{2,s}(\|g_k\|_0+\|g_j\|_0)\left(\|g_k\|_{\tla(\xi)|\xi|^{s}}+\|g_j\|_{\tla(\xi)|\xi|^{s}}\right)
    \end{align}
    where $\tla(\xi)$ is defined in \eqref{def tilde lambda}. In addition, $F_{1,s}$ and $F_{2,s}$ are analytic functions defined in a neighborhood of 0, admitting Taylor series expansions of the form
    \begin{gather}
        F_{1,s}(x)=xF_{2,s}(x),\\
        F_{2,s}(x)=\sum_{l=1}^{\infty}(2^sC_0)^lx^{l-1}.
    \end{gather}
\end{proposition}
\begin{proof}
    Proposition \ref{prop nonlinear kj} follows immediately as a corollary of Lemma \ref{lemma_technique_kj} and Lemma \ref{lemma_kj}.
\end{proof}
Following the same approach as in the proof of Proposition \ref{prop estimate N1}, we estimate each term in the Taylor series. Below is a technical lemma which will be used to derive the upper bound.
\begin{lemma}\label{lemma_technique_kj}
    The following inequality holds for all $\xi\in\mathbb{R}$ and $l\in\mathbb{N}^+$
    \begin{equation}
        \int_0^{\infty}|b_{l,k,j}(y)||e^{2i\xi y}-1|dy
\lesssim C_0^l\tla(\xi).
    \end{equation}
\end{lemma}
\begin{proof}
    Recall the estimate in  \eqref{upper bound bl} we are able to get
    \begin{equation}
        \int_0^{\infty}|b_{l,k,j}(y)||e^{2i\xi y}-1|dy\lesssim C_0^l\int_0^{\infty}\frac{y}{(y^2+1)^{l+1}}|e^{2i\xi y}-1|dy.
    \end{equation}
Then we see that
\begin{align*}
    \int_0^{\infty}\frac{y}{(y^2+1)^{l+1}}|e^{2i\xi y}-1|dy&\leq \int_0^{\infty}\frac{2y}{(y^2+1)^{l+1}}dy\\
    &=\frac{1}{l}\lesssim\tla(\xi)
\end{align*}
for $|\xi|>1$, and
\begin{align*}
    \int_0^{\infty}\frac{y}{(y^2+1)^{l+1}}|e^{2i\xi y}-1|dy&\leq\int_0^{\infty}|\xi|\frac{2y^2}{(y^2+1)^{l+1}}dy\\
    &=\frac{\pi}{2}\frac{(2l-3)!!}{(2l-2)!!}\frac{1}{l}|\xi|\lesssim \tla(\xi)
\end{align*}
for $|\xi|<1$ by noticing \eqref{def tilde lambda}. Thus this completes the proof.
\end{proof}
\begin{lemma}\label{lemma_kj}
    For each $l$ we have
    \begin{align}
        \int|\xi|^s&\left|\int b_{l,k,j}(y)\mathscr{F}(\partial_x[g_k,g_j]_y(x)([g_k,g_j]_y(x))^l)(\xi)dy\right|d\xi\\
       & \lesssim (2^{s+1}C_0)^l(\|g_k\|_0+\|g_j\|_0)^l\left(\|g_k\|_{\tla(\xi)|\xi|^{s+1}}+\|g_j\|_{\tla(\xi)|\xi|^{s+1}}\right)\\
        &+(2^{s+1}C_0)^l(\|g_k\|_1+\|g_j\|_1)(\|g_k\|_0+\|g_j\|_0)^{l-1}\left(\|g_k\|_{\tla(\xi)|\xi|^{s}}+\|g_j\|_{\tla(\xi)|\xi|^{s}}\right).
    \end{align}
\end{lemma}
\begin{proof}
    We first obtain
    \begin{equation}
        \mathscr{F}([g_k,g_j]_y)(\xi)=(\hat{g_k}-e^{-i\xi y}\hat{g_j})(\xi)=\m(\xi,y),\quad\mathscr{F}(\partial_x([g_k,g_j]_y))(\xi)=i\xi \m(\xi,y).
    \end{equation}
    Then follow the steps in the proof of Proposition \ref{prop estimate N1} we are able to write the Fourier transform as $l$ convolutions
    \begin{equation}
        \mathscr{F}(\partial_x([g_k,g_j]_y)([g_k,g_j]_y)^l)(\xi)=((i\xi \m)*(\m)*\dots*(\m))(\xi,y).
    \end{equation}
    Then we compute the left hand side of Lemma \ref{lemma_kj} as below by noting that $b_{l,k,j}(y)$ is odd with respect to $y$ in \eqref{property of bl}.
    \begin{align}
        &\int b_{l,k,j}(y)\mathscr{F}(\partial_x([g_k,g_j]_y)([g_k,g_j]_y)^l)(\xi)dy\\
    &=\int_0^{\infty}dy\int d\xi_1\cdots\int d\xi_l
    b_{l,k,j}(y)i(\xi-\xi_1)(\m(\xi-\xi_1,y)\cdots \m(\xi_l,y)-\m(\xi-\xi_1,-y)\cdots \m(\xi_l,-y)).
    \end{align}
    We can then express the difference term in the following form
    \begin{align}
        \m(\xi-\xi_1,y)\cdots \m(\xi_l,y)-\m(\xi-\xi_1,-y)&\cdots \m(\xi_l,-y)\\
        =\sum\limits_{q=0}^{l}\m(\xi_0-\xi_1,y)\dots \m(\xi_{q-1}-\xi_q,y)&\Big(\m(\xi_q-\xi_{q+1},y)-\m(\xi_q-\xi_{q+1},-y)\Big)\cdot\\
    &\m(\xi_{q+1}-\xi_{q+2},-y)\dots \m(\xi_l,-y)
    \end{align}
    with simplified notations $\xi_0=\xi$ and $\xi_{l+1}=0$. For arbitrary $\xi\in\mathbb{R}$, the following estimate holds
\begin{gather}
    |\m(\xi,y)|\leq(|\hat{g_k}|+|\hat{g_j}|)(\xi)\label{ineq_tildem1},\\
    |\m(\xi,y)-\m(\xi,-y)|\leq |e^{2i\xi y}-1||\hat{g_j}|(\xi)\leq|e^{2i\xi y}-1|(|\hat{g_k}|+|\hat{g_j}|)(\xi)\label{ineq_tildem2}.
\end{gather}
Integral the expression in Lemma \ref{lemma_kj} over $y$ by using the above estimate 
\eqref{ineq_tildem1}, \eqref{ineq_tildem2} and Lemma \ref{lemma_technique_kj} leads the following
\begin{align}
    &\int|b_{l,k,j}(y)\mathscr{F}(\partial_x([g_k,g_j]_y)([g_k,g_j]_y)^l)(\xi)|dy\\
    &\lesssim\int\dots\int  C_0^l(\tla(\xi-\xi_1)+\dots+\tla(\xi_l))|\xi-\xi_1|(|\hat{g_k}|+|\hat{g_j}|)(\xi-\xi_1)\dots(|\hat{g_k}|+|\hat{g_j}|)(\xi_l)d\xi_1\dots d\xi_l.
\end{align}
As in Lemma \ref{prop_estimate_N1_lemma} we employ the inequality
\begin{equation}
    |\xi|^s\leq(l+1)^{s-1}(|\xi-\xi_1|^s+\dots+|\xi_l|^s)
\end{equation}
and the estimate
\begin{equation}
    (\tla(\xi-\xi_1)+\dots+\tla(\xi_l))(|\xi-\xi_1|^s+\dots+|\xi_l|^s)\leq(l+1)(\tla(\xi-\xi_1)|\xi-\xi_1|^s+\dots+\tla(\xi_l)|\xi_l|^s)
\end{equation}
recalling $\tla(\xi)$ is non-decreasing as in \eqref{def tilde lambda}. Thus we obtain
\begin{align}
    &\int|\xi|^s\left|\int b_{l,k,j}(y)\mathscr{F}(\partial_x[g_k,g_j]_y(x)([g_k,g_j]_y(x))^l)(\xi)dy\right|d\xi\\
        &\lesssim C_0^l(l+1)^s\int\dots\int|\xi-\xi_1|(\tla(\xi-\xi_1)|\xi-\xi_1|^s+\dots+\tla(\xi_l)|\xi_l|^s)\\
        &\quad\times(|\hat{g_k}|+|\hat{g_j}|)(\xi-\xi_1)\dots(|\hat{g_k}|+|\hat{g_j}|)(\xi_l)d\xi_1\dots d\xi_ld\xi\\        
        & \lesssim (2^{s+1}C_0)^l(\|g_k\|_0+\|g_j\|_0)^l(\|g_k\|_{\tla(\xi)|\xi|^{s+1}}+\|g_j\|_{\tla(\xi)|\xi|^{s+1}})\\
        &\quad+(2^{s+1}C_0)^l(\|g_k\|_1+\|g_j\|_1)(\|g_k\|_0+\|g_j\|_0)^{l-1}(\|g_k\|_{\tla(\xi)|\xi|^{s}}+\|g_j\|_{\tla(\xi)|\xi|^{s}})
\end{align}
where the second inequality follows by expressing the integral in convolution form and reapplying inequality \eqref{ineq_convolution}.
\end{proof}

\subsection{Large time decay estimate}
To begin the analysis of large time decay, we introduce the simplified notation
\begin{equation}\label{notation EF}
    \begin{aligned}
        E_{0,s}(t)=\sum_{i=1}^n\|h_i(t)\|_{s},&\quad E_{1,s}(t)=\sum_{i=1}^n\|h_i(t)\|_{\tla(\xi)|\xi|^s},\\
        \TF_{0,s}(t)=F_{0,s}\left(\sum_{i=1}^n\|h_i(t)\|_1\right)=F_{0,s}\left(E_{0,1}(t)\right)&,\quad \TF_{1,s}(t)=F_{1,s}(E_{0,0}(t)),\quad \TF_{2,s}(t)=F_{2,s}(E_{0,0}(t)).
    \end{aligned}
\end{equation}
Recalling equation \eqref{eq vector g}, \eqref{def diagonalize} and \eqref{def h}  
 we obtain
\begin{equation}
    \frac{d}{dt}\vec{h}(\xi,t)=-|\xi|D(\xi)\vec{h}(\xi,t)+P^{-1}(\xi)\widehat{\vec{N}}(\xi,t).
\end{equation}
It follows
\begin{equation}\label{ineq begin hi}
    \begin{aligned}
    \frac{d}{dt}\|h_i\|_s&=\int|\xi|^s(\partial_t\hat{h}_i(\xi)\overline{\hat{h}_i(\xi)}+\hat{h}_i(\xi)\partial_t\overline{\hat{h}_i(\xi)})/(2|\hat{h}_i(\xi)|)d\xi\\
        &\leq-\int\lambda_i(\xi)|\xi|^{s+1}|\hat{h}_i(\xi)|d\xi+C\sum_{j,k=1}^n\int|\xi|^s|\mathscr{F}(N_{k,j})(\xi)|d\xi\\
        &\leq-\|h_i\|_{\tla(\xi)|\xi|^{s+1}}+C\sum_{j,k=1}^n\int|\xi|^s|\mathscr{F}(N_{k,j})(\xi)|d\xi.
\end{aligned}
\end{equation}
for each $1\leq i\leq n$. Then we set $a\delta=1$ where constant $a>\frac{1}{2}$ will be fixed later to obtain
\begin{gather}
    \frac{d}{dt}\left((1+\delta t)^a\|h_i\|_s\right)\leq(1+\delta t)^a\left((1+\delta t)^{-1}\|h_i\|_s+\frac{d}{dt}\|h_i\|_s\right),\label{ineq pre dt hi a}\\
    \frac{d}{dt}\left((1+\delta t)^{a-\frac{1}{2}}\|h_i\|_s\right)\leq(1+\delta t)^{a-\frac{1}{2}}\left(\frac{2a-1}{2a}(1+\delta t)^{-1}\|h_i\|_s+\frac{d}{dt}\|h_i\|_s\right).\label{ineq pre dt hi a-1}
\end{gather}
We decompose the integral for  $\|h_i\|_s$ into two parts $|\xi|<r$ and $|\xi|>r$:
\begin{equation}\label{ineq hi split a}
    \begin{aligned}
        \|h_i\|_s&\leq\int_{|\xi|<r}|\xi|^s|\hat{h_i}(\xi)|d\xi+\frac{2}{\alpha}\frac{1}{r^2}\int_{|\xi|>r}\tla(\xi)|\xi|^{s+1}|\hat{h_i}(\xi)|d\xi\\
        &\leq\int_{|\xi|\lesssim(1+\delta t)^{-\frac{1}{2}}}|\xi|^s|\hat{h_i}(\xi)|d\xi+\frac{1}{2}(1+\delta t)\|h_i\|_{\tla(\xi)|\xi|^{s+1}}\\
        &\leq C(1+\delta t)^{-\frac{s}{2}-\frac{1}{4}}\|h_i\|_{L^2}+\frac{1}{2}(1+\delta t)\|h_i\|_{\tla(\xi)|\xi|^{s+1}}.
    \end{aligned}
\end{equation}
Here we gain the first inequality by using the asymptotic behavior in \eqref{def tilde lambda}. The second inequality holds if we take $r=c(1+\delta t)^{-\frac{1}{2}}$ for some universal constant $c$ and the third inequality is due to Cauchy-Schwarz inequality with a universal constant $C$. For the bound of nonlinear terms we apply Proposition \ref{prop estimate N1} , Proposition \ref{prop nonlinear kj} and the notation \eqref{notation EF} to gain
\begin{equation}\label{ineq nonlinear cauchy}
    \begin{aligned}
    \sum_{j,k=1}^n\int|\xi|^s|\mathscr{F}(N_{k,j})(\xi)|d\xi\leq C\left( \TF_{0,s}E_{0,s+1}+\TF_{1,s}E_{1,s+1}+E_{0,1}\TF_{2,s}E_{1,s}\right)\\
    \leq \left(2\epsilon+C\TF_{1,s}+\frac{C}{\epsilon}\TF_{0,s}^2+\frac{C}{\epsilon}E_{0,1}^2\TF_{2,s}^2\right) E_{1,s+1}+\frac{C}{\epsilon}\left(\TF_{0,s}^2+E_{0,1}^2\TF_{2,s}^2\right)E_{0,s}.
\end{aligned}
\end{equation}
Here we use the Cauchy-Schwarz estimate
\begin{align}
    \TF_{0,s}E_{0,s+1}&=\TF_{0,s}\sum_{i=1}^n\int|\xi|^{s+1}|\hat{h}_i|d\xi\leq\sum_{i=1}^n\left(\epsilon\int\tla(\xi)|\xi|^{s+1}|\hat{h}_i|d\xi+\frac{C}{\epsilon}\TF_{0,s}^2\int\tla(\xi)^{-1}|\xi|^{s+1}|\hat{h}_i|d\xi\right)\\
    &\leq\sum_{i=1}^n\left(\epsilon\int\tla(\xi)|\xi|^{s+1}|\hat{h}_i|d\xi+\frac{C}{\epsilon}\TF_{0,s}^2\int\left(|\xi|^{-1}+\tla(\xi)\right)|\xi|^{s+1}|\hat{h}_i|d\xi\right)\\
    &\leq\sum_{i=1}^n\left(\left(\epsilon+\frac{C}{\epsilon}\TF_{0,s}^2\right)\int\tla(\xi)|\xi|^{s+1}|\hat{h}_i|d\xi+\frac{C}{\epsilon}\TF_{0,s}^2\int|\xi|^{s}|\hat{h}_i|d\xi\right)\\
    &=\left(\epsilon+\frac{C}{\epsilon}\TF_{0,s}^2\right) E_{1,s+1}+\frac{C}{\epsilon}\TF_{0,s}^2E_{0,s}
\end{align}
and
\begin{align}
    E_{0,1}\TF_{2,s}E_{1,s}&=E_{0,1}\TF_{2,s}\sum_{i=1}^n\int\tla(\xi)|\xi|^s|\hat{h}_i|d\xi\leq CE_{0,1}\TF_{2,s}\sum_{i=1}^n\int|\xi|^{s+1}|\hat{h}_i|d\xi\\
    &\leq\left(\epsilon+\frac{C}{\epsilon}E_{0,1}^2\TF_{2,s}^2\right) E_{1,s+1}+\frac{C}{\epsilon}E_{0,1}^2\TF_{2,s}^2E_{0,s}
\end{align}
where we use the properties of $\tla(\xi)$:
\begin{equation}
    \tla(\xi)^{-1}\lesssim|\xi|^{-1}+\tla(\xi),\quad \tla(\xi)\lesssim|\xi|
\end{equation}
by recalling \eqref{def tilde lambda}.
Combining the inequalities \eqref{ineq begin hi}, \eqref{ineq pre dt hi a}, \eqref{ineq pre dt hi a-1}, \eqref{ineq hi split a} and \eqref{ineq nonlinear cauchy} by choosing a sufficiently small constant $\epsilon$ we get
\begin{equation}\label{ineq dt hi a}
    \begin{aligned}
    \frac{d}{dt}&\left((1+\delta t)^aE_{0,s}\right)\\
    \leq &C(1+\delta t)^{a-\frac{s}{2}-\frac{5}{4}}\sum_{i=1}^n\|h_i\|_{L^2}-(1+\delta t)^a\left(\frac{1}{4}-C\left(\TF_{1,s}+\TF_{0,s}^2+E_{0,1}^2\TF_{2,s}^2\right)\right)E_{1,s+1}\\
    &+(1+\delta t)^aC\left(\TF_{0,s}^2+E_{0,1}^2\TF_{2,s}^2\right)E_{0,s}
\end{aligned}
\end{equation}
and similarly
\begin{equation}\label{ineq dt hi a-1}
    \begin{aligned}
    \frac{d}{dt}&\left((1+\delta t)^{a-\frac{1}{2}}E_{0,s}\right)\\
    \leq &C(1+\delta t)^{a-\frac{1}{2}-\frac{s}{2}-\frac{5}{4}}\sum_{i=1}^n\|h_i\|_{L^2}-(1+\delta t)^{a-\frac{1}{2}}\left(\frac{1}{4}-C\left(\TF_{1,s}+\TF_{0,s}^2+E_{0,1}^2\TF_{2,s}^2\right)\right)E_{1,s+1}\\
    &+(1+\delta t)^{a-\frac{1}{2}}C\left(\TF_{0,s}^2+E_{0,1}^2\TF_{2,s}^2\right)E_{0,s}.
\end{aligned}
\end{equation}
These computations lead to the following proposition.
\begin{proposition}\label{prop decay s01}
    There exists a universal constant $\gamma$ such that if the initial data satisfies 
\begin{equation}
    \sum_{j=1}^{n}\left(\|\tilde{g}_j\|_{L^2}+\|\tilde{g}_j\|_0+\|\tilde{g}_j\|_1\right)\leq\gamma,
\end{equation}
then we have large time decay estimates
\begin{equation}
    \sum_{j=1}^n\|g_j(t)\|_0\lesssim(1+t)^{-\frac{1}{4}},\quad \sum_{j=1}^n\|g_j(t)\|_1\lesssim(1+t)^{-\frac{3}{4}}.
\end{equation}
\end{proposition}
\begin{proof}
    By Remark \ref{remark gh} it is equivalent to deal with $h_i$. We aim to control the following
    \begin{equation}
        \TE(t):=(1+\delta t)^aE_{0,1}(t)+(1+\delta t)^{a-\frac{1}{2}}E_{0,0}(t)=(1+\delta t)^a\sum_{i=1}^n\|h_i\|_1+(1+\delta t)^{a-\frac{1}{2}}\sum_{i=1}^n\|h_i\|_0
    \end{equation}
    with bootstrap assumption
    \begin{equation}\label{bootstrap assume}
        \TE(t)\leq c_0(1+\delta t)^{a-\frac{3}{4}}
    \end{equation}
    for a small constant $c_0$. This leads to
    \begin{equation}
        E_{0,1}(t)\leq c_0,\quad E_{0,0}(t)\leq c_0
    \end{equation}
    which implies we can choose $c_0$ small enough such that
    \begin{equation}
        \frac{1}{4}-C\left(\TF_{1,s}+\TF_{0,s}^2+E_{0,1}^2\TF_{2,s}^2\right)\geq0.
    \end{equation}    
    Then using \eqref{ineq dt hi a} and \eqref{ineq dt hi a-1} we get
    \begin{equation}
        \begin{aligned}
            \frac{d}{dt}\TE(t)\lesssim&(1+\delta t)^{a-\frac{7}{4}}\sum_{i=1}^n\|h_i\|_{L^2}+\left(\TF_{0,s}^2+E_{0,1}^2\TF_{2,s}^2\right)\TE
            \\
            \lesssim&(1+\delta t)^{a-\frac{7}{4}}\sum_{i=1}^n\|h_i\|_{L^2}+\left(E_{0,1}^4+E_{0,1}^2\right)\TE\\
            \lesssim&(1+\delta t)^{a-\frac{7}{4}}\sum_{i=1}^n\|h_i\|_{L^2}+(1+\delta t)^{-2a}(\TE^4+\TE^2)\TE
        \end{aligned}
    \end{equation}
    where in the second inequality we utilize the asymptotic approximations that $F_{0,s}\approx x^2,F_{2,s}\approx 1$ near $x=0$. To arrive at this conclusion we set $a>\frac{3}{4}$ first, using assumption \eqref{bootstrap assume}
    \begin{align}
        \frac{d}{dt}\TE(t)&\lesssim(1+\delta t)^{a-\frac{7}{4}}\sum_{i=1}^n\|h_i\|_{L^2}+(1+\delta t)^{-2a+5(a-\frac{3}{4})}(c_0^5+c_0^3)\\
        &\leq(1+\delta t)^{a-\frac{7}{4}}\left(\sum_{i=1}^n\|h_i\|_{L^2}+c_0^5+c_0^3\right)
    \end{align}
    where the second inequality holds if we take $a-\frac{7}{4}\leq-2a+5(a-\frac{3}{4})$, as well as $a\leq 1$. Therefore we take arbitrary $a$ satisfying $\frac{3}{4}<a\leq1$ with $\delta=a^{-1}$. Theorem \ref{energy conservation} provides $\sum_{i=1}^n\|h_i\|_{L^2}\lesssim\gamma$. Hence integral the above expression we gain
    \begin{align}
         \TE(t)&\lesssim(1+\delta t)^{a-\frac{3}{4}}\left(\gamma+(1+\delta t)^{-a+\frac{3}{4}}\TE(0)+c_0^5+c_0^3\right)\\
        &\leq (1+\delta t)^{a-\frac{3}{4}}(2\gamma+c_0^5+c_0^3).
    \end{align}
    This implies we can first set $\gamma$ is uniformly small such that there exists small $c_0>0$ satisfying 
\begin{equation}
    2\gamma+c_0^5+c_0^3\leq C^{-1} c_0
\end{equation}
   and then let $\gamma$ to satisfy $\TE(0)\leq2\gamma\leq c_0$. Therefore we have closed the bootstrap and proved the proposition.
\end{proof}

\begin{corollary}
    For any constant $\gamma>0$, there exists $\gamma_0>0$, such that if
    \begin{equation}
        \sum_{j=1}^n\|\tilde{g}_j\|_{L^2}+\|\tilde{g}_j\|_0+\|\tilde{g}_j\|_1\leq \gamma_0,
    \end{equation}
    then
    \begin{equation}
        \sum_{j=1}^n\|g_j(t)\|_{L^2}+\|g_j(t)\|_0+\|g_j(t)\|_1\leq \gamma
    \end{equation}
    holds for all time $t$.
\end{corollary}
We are now ready to prove Theorem \ref{theorem main}.
\begin{proof}[Proof of Theorem \ref{theorem main}]
    Choose a uniform small $\gamma$ such that Proposition \ref{prop decay s01} holds. Use \eqref{ineq dt hi a} and \eqref{ineq dt hi a-1} again to obtain
    \begin{align}
        \frac{d}{dt}\left((1+\delta t)^aE_{0,s}(t)\right)
        &\lesssim (1+\delta t)^{a-\frac{s}{2}-\frac{5}{4}}\sum_{i=1}^n\|h_i\|_{L^2}+\left(F_{0,s}^2+E_{0,1}^2F_{2,s}^2\right)(1+\delta t)^aE_{0,s}\\
        &\lesssim (1+\delta t)^{a-\frac{s}{2}-\frac{5}{4}}\sum_{i=1}^n\|h_i\|_{L^2}+\left(E_{0,1}^4+E_{0,1}^2\right)(1+\delta t)^aE_{0,s}\\
        &\lesssim (1+\delta t)^{a-\frac{s}{2}-\frac{5}{4}}\sum_{i=1}^n\|h_i\|_{L^2}+(1+\delta t)^{-\frac{3}{2}}(1+\delta t)^aE_{0,s}.
    \end{align}
   Note $(1+\delta t)^{-\frac{3}{2}}$ is integrable, thus we apply Gronwall inequality by setting $a>\frac{s}{2}+\frac{1}{4}$ to obtain
   \begin{equation}
       (1+\delta t)^aE_{0,s}(t)\lesssim E_{0,s}(0)+\int_0^t(1+\delta t)^{a-\frac{s}{2}-\frac{5}{4}}dt\lesssim (1+\delta t)^{a-\frac{s}{2}-\frac{1}{4}}
   \end{equation}
which finishes the proof.
\end{proof}

\vspace{0.2in}
\noindent
{\bf Acknowledgments}. The author is partially supported by the National Key Research and Development Program of China under the grant No. 2021YFA1001500. He thanks Jiajun Tong for helpful educational discussions and comments.
\bibliographystyle{plain} 
\bibliography{refs} 

\end{document}